\documentclass[12pt]{amsart}%

\usepackage{amsmath, amssymb, amsthm, amsfonts, bbm, dsfont}
\usepackage{graphicx}
\usepackage{float}
\usepackage{wrapfig}
\restylefloat{figure}
\usepackage{mathrsfs}

\usepackage{hyperref}

\usepackage{a4wide}

\numberwithin{equation}{section}
\theoremstyle{plain}
\newtheorem{theorem}[subsubsection]{Theorem}
 \newtheorem{lemma}[subsubsection]{Lemma}
 \newtheorem{proposition}[subsubsection]{Proposition}

 \theoremstyle{definition}
 \newtheorem{defn}[subsubsection]{Definition}

\newtheorem{remark}[subsubsection]{Remark}
\newtheorem{exam}[subsubsection]{Example}

\usepackage{stackengine}
\usepackage{mathtools}
\usepackage{amssymb}
\usepackage{dsfont}

\usepackage{mathabx}

\newcommand{\CC}{\mathbb{C}}

\newcommand{\QQ}{\mathbb{Q}}
\newcommand{\ZZ}{\mathbb{Z}}

\newcommand{\calO}{\mathcal{O}}

\newcommand\rk{\textup{rk}}

\newcommand\Aut{\textup{Aut}}
\newcommand\Hom{\textup{Hom}}

\newcommand{\quash}[1]{}

\newcommand{\hk}{hyperk\"{a}hler }

\newcommand{\ktipo}{$K3^{[2]}$ type}
\newcommand{\kntipo}{$K3^{[n]}$ type}
\newcommand{\kntiposp}{$K3^{[n]}$ type }

\usepackage{tikz}
\usetikzlibrary{shapes,fit,intersections}
\usetikzlibrary{decorations.markings}
\usetikzlibrary{decorations.pathreplacing}
\usetikzlibrary{patterns}
\usetikzlibrary{matrix,arrows}
\usepgflibrary{decorations.pathmorphing}

\usepackage{tikz-cd}

\title[Fixed loci of symplectic automorphisms]{Fixed loci of symplectic automorphisms of \kntiposp and 
\(n\)-Kummer type manifolds}
\author{Ljudmila Kamenova}

\author{Giovanni Mongardi}

\author{Alexei Oblomkov}

\address{Department of Mathematics, Stony Brook University, Math Tower 3-115, 
Stony Brook, NY 11794-3651, USA}
\email{kamenova@math.stonybrook.edu}

\address{Dipartimento di Matematica, Università degli studi di Bologna,  
Piazza di porta san Donato 5, Bologna, Italia 40126  }
\email{giovanni.mongardi2@unibo.it}

\address{Department of Mathematics and Statistics, University of Massachusetts, Lederle Graduate Tower, Amherst, MA 01003, USA}
\email{oblomkov@math.umass.edu}

\begin{document}

\begin{abstract}
The aim of this paper is to give an explicit description of the fixed loci 
of symplectic automorphisms for certain \hk manifolds, namely for Hilbert 
schemes on K3 surfaces and for generalized Kummer varieties. Here we extend 
our previous results from the case of involutions to more general groups. 
In particular, under some conditions on the dimension, we give the full 
answer for finite group actions of symplectic automorphisms coming from K3 
surfaces. We prove that the all irreducible components of the fixed loci 
are of $K3^{[k]}$ type of lower dimensions or isolated points. 
\end{abstract}

\maketitle

\section{Introduction}

Consider a K3 surface $S$ together with a symplectic action of a finite 
group $G$ on $S$. The action of $G$ is called {\it symplectic} if the 
induced action on $H^{2,0}(S)$ is trivial, i.e., the 
resolution of the quotient $S/G$ is a K3 surface. 
Following the works of Nikulin and of Mukai, in \cite{xiao} Xiao gives a 
combinatorial 
classification of finite groups $G$ admitting a symplectic action on $S$ 
together with the number of fixed points of each type. The types of fixed 
points of the action of $G$ on $S$ correspond to the types of 
singularities of the quotient $S/G$. An analogous classification for Abelian surfaces was given in \cite{Fuj88}.

Let $X=S^{[n]}$ be the punctual Hilbert scheme on a K3 surface $S$ with a 
symplectic action of $G$ induced from a symplectic action of $G$ on $S$. 
In this paper we classify the fixed loci of this action, and we give 
formulas for the number of connected components for each possible dimension.

In order to describe the irreducible components of the fixed locus of 
the \(G\) action on  \(S^{[n]}\) we use the methods of \cite{BryanGyenge22}. 
Additionally, we are also able to describe the connected components of 
generalized Kummer varieties by considering singular \(K3\) surfaces. 
Using the computations in Section \ref{sec:irred}, we are able to determine 
the fixed locus for all standard automorphisms, and we obtain the following 
main results.  

\begin{theorem} \label{thm:1}
Let $S$ be a K3 surface and let $G$ be a finite symplectic group of 
automorphisms of $S$. Let $S^{[n]}$ be the Hilbert scheme of $n$ points 
on $S$ and let us consider the induced action of $G$ on $S^{[n]}$. 
Then all the irreducible components of the locus stabilized by $G$ 
are deformation of Hilbert schemes of points on a K3 surface or isolated points, 
and their number for each dimension $2k$ is 
$$ N_k= \Theta_G[n-|G|k],$$
where $\Theta_G[m]$ is the $m$-th coefficients of the Theta series 
\eqref{eq:theta} of the lattice \(M^G\).
\end{theorem}

If we assume that $n-1$ is a prime power, then any numerically standard 
group in cohomology becomes standard by Theorem \ref{thm:main_old}, and 
therefore we can deform the pair $(X,G)$ to a natural pair $(S^{[n]},G)$. 
Hence, the following main result follows from the previous theorem.

\begin{theorem} \label{thm:2}
Let $X$ be a \hk manifold of \kntiposp such that $n-1$ is a prime power. 
Let $G\subset Aut(X)$ be a finite group of symplectic automorphisms which 
is numerically standard in cohomology.  
Then all irreducible components of $X^G$ are of $K3^{[k]}$ type, and their 
number is
$$ N_k= \Theta_G[n-|G|k],$$
where $\Theta_G[m]$ is the $m$-th coefficients of the Theta series 
\eqref{eq:theta} of the lattice \(M^G\).
\end{theorem}

Instead of restricting $n$, we could impose a restriction on the group 
of automorphisms $G$. By Proposition \ref{prop:goodgroups}, all numerically 
standard groups listed in Table 1 
 are standard, and 
therefore we can deform the given pair $(X,G)$ to a natural pair 
$(S^{[n]},G)$. The following result would then follow from 
Theorem \ref{thm:mainK3}.

\begin{theorem}\label{thm:3}
Let $X$ be a \hk manifold of \kntiposp and let $G\subset Aut(X)$ be a 
finite group of symplectic automorphisms among those listed in Table 1 
such that it is numerically standard in cohomology.
Then all irreducible components of $X^G$ are of $K3^{[k]}$ type, and their 
number is
$$ N_k= \Theta_G[n-|G|k],$$
\end{theorem}

For generalized Kummer varieties we have a similar statement assuming that 
$G$ does not contain translations by torsion points on $A$, nor involutions obtained as a composition of a translation by a point of order at least three and a sign change. A group satisfying these conditions will be called a \emph{regular} group. 

\begin{theorem} 
Let $X$ be a \hk manifold of \(n\)-Kummer type and let $G\subset Aut(A)$ 
be a finite group of symplectic automorphisms of the abelian surface $A$.
Assume $G$ preserves the Albanese map of the generalized Kummer $X$ (i.e. commutes with Albanese map and preserves its zeroth fiber), 
and is a regular group. 
Then all irreducible components of $X^G$ are of $K3^{[k]}$ type, and their 
number is 
$$ N_k= \Theta_G[n-|G|k;1],$$
where $\Theta_G[m;1]$ is the coefficient of the Theta series \eqref{eq:theta2} detailed in \eqref{eq:theta2-coeffs}.
\end{theorem}
If the group $G$ is non regular, we could obtain fixed loci of 
\(k\)-Kummer type for some $k \leq \frac{n}{2}$. For example, if $n$ is odd 
and $G = \ZZ_2$, where the involution $\iota$ generating $G$ is induced by 
the shift by an order two torsion point, then the fixed locus of $\iota$ 
consists of $8$ copies of a generalized Kummer variety of complex dimension 
$n-1$. 

A different approach to the determination of fixed loci was taken by Beckman 
and Oberdieck in \cite{BO}: their approach uses the equivariant derived 
category and determines the geometry of each component of the fixed locus 
under some mild hypothesis on the group of derived equivalences of a $K3$ 
surface. Although this applies to the groups we consider for the case of 
manifolds of \kntipo, their results cannot be used in general to compute the number 
of these components, which is the main result of this paper. This is due 
to the fact that they obtain an étale map from some moduli spaces of 
objects in the equivariant category to the fixed locus of the group 
action on the initial manifold, but the geometry of the étale map 
cannot be precisely determined for arbitrary groups. If the group is however cyclic, by \cite[Theorem 1.4]{BO} the geometry of this étale map is clear and a computation of the number of components of the fixed locus can be obtained also with their method.

\section{Admissible groups}\label{sec:adm}
In this section we review some of the results of \cite{KamenovaMongardiOblomkov19} concerning deformations of pairs $(X,G)$, where $X$ is a \hk manifold and $G\subset Aut_s(X)$ is a subgroup of symplectic automorphisms. In particular, we give a necessary and sufficient condition to ensure that such a pair can be deformed to a pair $(S^{[n]},G)$, where $G\subset Aut_s(S)$ and the action on the Hilbert scheme is induced by that on the $K3$ surface. This will allow us to compute the fixed locus of $G$ in all deformations of such a pair. 
Moreover, we give an easier method to compute a sufficient condition, which will be used for some groups and dimensions to obtain the same result. 

We are interested in manifolds of \kntiposp or of Kummer type and in their second cohomology. If $X$ is of \kntipo, we have
$$H^2(X,\ZZ)=U^3\oplus E_8(-1)^2\oplus (2-2n), $$
where $H^2$ is endowed with a lattice structure given by the Beauville-Bogomolov-Fujiki form, all sums are orthogonal, $U$ is the hyperbolic plane, $E_8(-1)$ is the unimodular even negative definite lattice of rank 8 and $(k)$ denotes a rank one lattice with a generator of square $k$. By work of Markman \cite[Section 9]{mar_tor}, there is a canonically defined embedding (up to isometry)
$$H^2(X,\ZZ)\to U^4\oplus E_8(-1)^2, $$
which is invariant under parallel transport. The lattice on the right hand side is usually called (due to the constructions as moduli spaces of sheaves) the Mukai lattice associated to $X$.

If $X$ is of Kummer n type, we have
$$H^2(X,\ZZ)=U^3\oplus (-2-2n), $$
and, by work of Wieneck \cite[Theorem 4.1]{wie}, there is again a canonically defined embedding (up to isometry)
$$H^2(X,\ZZ)\to U^4,$$
where the latter is also called the Mukai lattice associated to $X$.

\begin{defn} 
Let $S$ be a $K3$ surface and let $G\subset \mathrm{Aut}_s(S)$ be a subgroup 
of the symplectic automorphisms on $S$. Then $G$ induces a subgroup of the 
symplectic morphisms on $S^{[n]}$ which we still denote by $G$. We call the 
pair $(S^{[n]},G)$ a natural pair as in \cite[Definition 1]{boi}. 
The pair $(X,H)$ is standard if it is deformation equivalent to a natural pair, as in \cite[Definition 1.2]{mon1}. 
If $A$ is an abelian surface, the same definitions apply to the generalized Kummer $2n$-fold $K_n(A)$ and symplectic automorphisms preserving the sum of $n+1$ points of $A$, however the reader should notice that the induced action of $G$ on $H^2(K_n(A))$ is not necessarily faithful (while it stays faithful on $K_n(A)$ and $H^*(K_n(A)))$, as there is a group of automorphisms acting trivially on the second cohomology, see \cite{BNS_enriques} .
\end{defn}

\begin{remark}
Notice that we stick with the convention used in \cite[Definition 1.2]{mon1}, where a pair $(X,G)$ is called standard if it can be deformed to a natural pair, instead of using the definition from \cite[Definition 4.1]{BCS14}, where the authors call such a pair natural as well, but notice that the two definitions are actually equivalent. 
\end{remark}

\begin{defn}
Let $G$ be a finite group acting faithfully on a manifold $X$. 
Define the invariant lattice $T_G(X)$ inside $H^2(X, \mathbb Z)$ to be the 
fixed locus of the induced action of $G$ on the second cohomology. The co-invariant 
lattice $S_G(X)$ is the orthogonal complement $T_G(X)^\perp$. The fixed locus 
of $G$ on $X$ is denoted by $X^G$. 
\end{defn}

As automorphisms of K3 and abelian surfaces are better known, it is interesting to determine whether an automorphism group on a manifold of \kntiposp (or of Kummer $n$ type) is standard or not. For related works in this direction, see \cite{BCS14} and \cite{jou}. Following \cite{KamenovaMongardiOblomkov19}, we give the following criterion:

\begin{defn}\label{num_stand}
Let $Y$ be a manifold of \kntiposp or of Kummer $n$ type. A pair $(Y,H)$ is called numerically standard if the representation of $H$ on $H^2(Y,\mathbb{Z})$ is isomorphic to that of a standard pair $(X,H)$, up to conjugation by the monodromy group. More precisely, there exists a $K3$ (or abelian) surface $S$ with an $H$ action such that
\begin{itemize}
\item $S_H(S)\cong S_H(Y)$,
\item $T_H(S)\oplus \mathbb{Z}\delta=T_H(S^{[n]})\cong T_H(Y)$, where $2\delta$ is the class of the exceptional divisor of $S^{[n]}\rightarrow S^{(n)}$ (and analogously for the Kummer $n$ case),
\item The two isomorphisms above extend to isomorphisms of the Mukai lattices $U^4\oplus E_8(-1)^2$ (or $U^4$ in the Kummer case) after taking the canonical choice of an embedding of $H^2$ into the Mukai lattice.
\end{itemize}
All the above isomorphisms are $H$-equivariant.
When only the first two conditions are satisfied, we will call such a group numerically standard in cohomology.
\end{defn}

Notice that the first two conditions in the definition only amount to asking that the action of $H$ on the second cohomology coincides with the action of a standard pair.

\begin{remark}
There are several examples of groups of automorphisms which are not numerically standard, nor numerically standard in cohomology. The easiest examples to be found in the literature are symplectic automorphisms of order 11, as described in \cite[Thm 1.3]{mon4}: this cannot be standard in any sense, as there is no order 11 action on any K3 surface.

Another interesting case is given by special order three automorphisms acting on the Fano variety of lines of a cubic fourfold whose equation splits as the sum of two degree three equations $f(x_0,x_1,x_2)+g(x_3,x_4,x_5)$, each involving only three different coordinates. In this case we have two disjoint planes $P_0=\{x_3=x_4=x_5=0\},\,P_1=\{x_0=x_1=x_2=0\}$ and the intersection of the cubic with them is given by the elliptic curves $V(f)\subset P_0$ and $V(g)\subset P_1$. The action is by multiplication by a third root on unity on $P_0$ and trivial on $P_1$. The fixed locus on the Fano varieties of lines consists in the abelian surface obtained by taking lines with one point in $V(f)$ and one in $V(g)$. The action in cohomology has a fixed lattice of rank $5$. These automorphisms were studied first by Namikawa in \cite[1.7 (iv)]{Nami}. In this case, the quotient of the Fano variety of lines has an $A_2$ singularity along the fixed abelian surface $V(f)\times V(g)$, and the minimal resolution of it is a generalized Kummer manifold, obtained from the same abelian surface (see \cite{kawa} for further details)  
\end{remark}

The key result of \cite{KamenovaMongardiOblomkov19} is the following:

\begin{theorem} \label{thm:main_old}
Let $X$ be a manifold of \kntiposp or of Kummer $n$ type. Let $G\subset \mathrm{Aut}_s(X)$ be a 
finite group of numerically standard automorphisms. Then $(X,G)$ is a 
standard pair.
\end{theorem}

\begin{remark}
When $n-1$ is a prime power, the third condition in Definition \ref{num_stand} is unnecessary, as there is only one isometry orbit of embeddings of $H^2(X,\ZZ)$ in its Mukai lattice, so that the statement of Theorem \ref{thm:main_old} coincides with the main result of \cite{mon2}.
\end{remark}

When $n-1$ is not a prime power, in \cite{KamenovaMongardiOblomkov19} a technical condition on the invariant latticed was entroduced in order to ensure that a group can be deformed to a group acting naturally on a Hilbert scheme of points. Let us briefly recall it:

\begin{proposition}\label{prop:mukai_not}
Let $(X,G)$ be a pair such that there exists a $K3$ (resp. Abelian) surface $S$ and $G\subset Aut_s(S)$ such that $H^2(S^{[n]})$ (resp. $H^2(K_n(A))$) and $H^2(X)$ are isomorphic $G$ representations. Moreover, suppose that $U\subset T_G(S)$. Then $(X,G)$ is numerically standard.
\end{proposition}

\begin{remark}
Notice that the above criterion is sufficient but not necessary, indeed if $n-1$ is a prime power the condition that $H^2(S^{[n]})$ and $H^2(X)$ are isomorphic $G$ representation suffices. Moreover, this condition is used to ensure that the invariant lattice for the $G$ action on the Mukai lattice of $X$ contains two copies of the hyperbolic lattice $U$, which is used to ensure that the isomorphic $G$ actions on second cohomology can be extended to isomorphic $G$ actions on the Mukai lattices. Essentially, what we are using is that embeddings of a vector $v$ of square $2d$ are not necessarily unique up to isometry in the lattice $U$, while they become unique in $U^2$, where all elements of a given square are isometry equivalent by Eichler's criterion \cite[Proposition 3.3]{ghs}. 
\end{remark}

We wish to apply the above to finite groups acting on K3 surfaces, and we obtain the following result.

\begin{proposition}\label{prop:goodgroups}
Let $G$ be a finite group acting symplectically on a K3 surface. Then the 
conditions of Proposition \ref{prop:mukai_not} are satisfied for all cases 
in Table 1 (12 out of 81 cases)
\end{proposition}

\begin{proof}
The proof is straigthforward, as it follows immediately from Xiao's classification of group actions \cite{xiao} together with the computation of the $G$ invariant lattice by Hashimoto \cite{hashi}. We include in table 
3 all the groups and lattices involved.
\end{proof}

\begin{table}[h!]\label{tab:goodgroups}
\begin{tabular}{|c|c|c|c|}
\hline
Label & Group & $T_G(K3)$& singularities in the quotient\\ \hline
1 & $C_2$ & $U^3\oplus E_8(-2)$& \(8\, A_1\)\\ \hline
2 & $C_3$ & $U\oplus U(3)\oplus A_2(-1)^2$& \(6\,A_2\)\\ \hline
3 & $C_2^2$ & $U\oplus U(2)^2\oplus D_4(-2)$& \(12 A_1\)\\ \hline
4 & $C_4$ & $U\oplus U(4)^2\oplus -2^2$& \(4A_3 + 2A_1\) \\ \hline
5 & $C_{5}$ & $U\oplus U(5)^2$& \(4A_4\)\\ \hline
7 & $C_{6}$ & $U\oplus U(6)^2$& \(2A_5+2A_2+2A_1\)\\ \hline

10 & $D_8$ & $U\oplus 4^2 \oplus -4^3$& \(2A_3+9A_1\) \\ \hline
16 & $D_{10}$ & $U\oplus U(5)^2$& \(2A_4+8A_1\)\\ \hline
17 & $A_4$ & $U\oplus A_2(2)\oplus A_2(-4)$& \(6A_2+4A_1\) \\ \hline
18 & $D_{12}$ & $U\oplus U(6)^2$& \(A_5+A_2+9A_1\)\\ \hline
34 & $S_4$ & $U\oplus A_2(4)\oplus -12$& \(2A_3+3A_2+2A_1\) \\ \hline
  55 & $A_5$ & $U\oplus A_2(10)$& \(2A_4+3A_2+4A_1\)\\ \hline

\end{tabular}
\vspace{10pt}

\caption{Automorphisms with at least one copy of $U$ in $T_G(K3)$}
\end{table}

The last column on the table indicates the type of the singularities that appears 
in the local analysis of the fixed locus of the action of the group.

\newpage

\section{Computations of fixed loci}
\label{sec:irred}
In this section we show the enumerative part of  theorems~\ref{thm:1},\ref{thm:2},\ref{thm:3}. We start with a general discussion of the group
action on Hilbert schemes of symplectic surfaces.

To describe the irreducible components of the fixed locus of the \(G\) action on  \(S^{[n]}\) we  use the method of
\cite{BryanGyenge22}. First, let us introduce some notations. Let \(p_i\in S\), \(i=1,\dots,k\) be points with non-trivial
stabilizers \(\tilde{G}_i\subset G\).  Also \(q_i\in S/G\), \(i=1,\dots,\ell\)  are the orbifold points of \(S/G\) and \(G_i\) is the corresponding
orbifold group. It is more descriptive to use notation \(Sing(S/G)\) for the set \(\{1,\dots,\ell\}\). Let us point out that there are many cases when
\(k>\ell\), in particular each orbifold point \(q_j\) correspond to a \(G\)-orbit of some point \(p_i\) with a non-trivial stabilizer.

To explain the key combinatorial result  we need to fix some representation theoretic notations.
Let us first discuss the local situation. Let \(G_\Delta\subset SU(2)\) be a McKay subgroup corresponding to the Dynkin diagram
\(\Delta\). Nodes of the diagram correspond to non-trivial irreducible representations \(\rho_j\), \(j=1,\dots,r\).
Let us use notation  \(\rho_0\) for the trivial representation and  \(\rho_{reg}\) for the regular representation.
In particular, \([\rho_{reg}]=\sum_j d^j[\rho_j]\), \(d^j=d^j(\Delta)=\dim \rho^j\).

Let \(M_\Delta\)  be the root lattice of \(\Delta\). It is the \(\ZZ\)-span of \(\rho_j\), \(j=1,\dots,r\) and the Dynkin pairing is
defined by \((\rho_j,\rho_j)_{\Delta}=-2\), \((\rho_i,\rho_j)_{\Delta}=\dim \Hom_G(\rho_j,\rho_i\otimes \CC^2)\), \(i\ne j\).
The last number is \(1\) if \(i\) and \(j\) are connected by an edge of \(\Delta\) and zero otherwise.  Let us also fix the notation \(\tilde{\Delta}_i\) and \(\Delta_i\) for the
Dynkin graph of \(\tilde{G}_i\) and \(G_i\), respectively.

Let \(\mathcal{T}=\pi_*(\mathcal{O}/\mathcal{I})\) be the tautological vector bundle, where \(\pi:S^{[n]}\times S\to S^{[n]}\) and \(\mathcal{I}\) is a universal ideal sheaf. 
Let \(U_i\) be an \(G_i\)-equivariant affine chart around \(p_i\) that does not contain any other points with  nontrivial stabilizer. The rank of \(\mathcal{T}_{U_i}=\calO_{U_i}/\mathcal{I}_{U_i}\) is
upper semi-continuous on \(S^{[n]}\).


Let \(I\in (S^{[n]})^G\) then    we have the following decomposition in the ring of the virtual \(\tilde{G}_i\)-representations:
\[(\mathcal{T}_{U_i})_I=(\hat{m}^0_i(I)[\rho_{reg}]+\sum_{j=1}^{r_i}m_i^j(I)[\rho_j])\otimes \mathcal{O}|_{V_i(I)}.\]

The number \(\hat{m}^0_i(I)\) varies along the connected components of \((S^{[n]})^G\)  but  the numbers \(m_i^j(I)\) are locally constant.
Indeed, let \(\phi: (D,0)\to (S^{[n]})^G\) be a non-constant morphism of a formal disc such that \(\rk((\mathcal{T}_{U_i})_{\phi(0)})>
\rk((\mathcal{T}_{U_i})_{\phi(\eta)})\), here \(\eta\) is a generic point of \((D,0)\). That means \(\phi\) is an arc in \((S^{[n]})^G\) such that
the end of the arc contains the clusters of points in the compliment \(S\setminus U_i\) and the clusters are the limits of the clusters
in \(U_i\). In particular, these clusters are far away from the point \(p\). Since \(U_i\setminus p_i\) is a \(G_i\) torsor the virtual \(G_i\) representation
\((\mathcal{T}_{U_i})\) can only change by a multiple of \(\rho_{reg}\).

Similarly, as we vary \(I\) along \((S^{[n]})^G\) the numbers \(m_i^j(I)\) do not change along a connected component of
\((S^{[n]})^G\).  Thus we can define an invariant of
an ideal \(I\in (S^{[n]})^G\), \(\mathbf{m}(I)=(m_i^j(I))\), \(i=1,\dots,k\), \(j=1,\dots,r_i\). The subspace of ideals with fixed
\(\mathbf{m}\) is denoted by:
\[S^{[n,\mathbf{m}]}=\{I\in (S^{[n]})^G|\,\mathbf{m}(I)=\mathbf{m}\}\]

The group \(G\) acts on the collection of fixed points \(p_i\) by permuting them and orbits correspond to the orbifold points \(q_j\in S/G\).
Thus we have an action of \(G\) on the product \(\prod_{i=1}^k M_{\tilde{\Delta}_i}\) by permuting corresponding factors of the product.
Hence we have a subspace of \(G\)-invariants:
\begin{equation}\label{eq:MM}
  M=\prod_{i=1}^k M_{\tilde{\Delta}_i},\quad M^G\subset M.
\end{equation}
Requiring that an ideal $I$ lies in the fixed locus $(S^{[n]})^G$ is equivalent to asking that the datum \(\mathbf{m}(I)\) must be \(G\)-invariant.
Let us also introduce a slice to the \(G\)-action:
\begin{equation}\label{eq:MM-1}
  M_G=\prod_{q_i\in Sing(S/G)} M_{\Delta_i}.
\end{equation}  
The lattice \(M_G\) has natural pairing induced from the factors and a natural embedding \(i: M_G\to M^G\subset M\).

\begin{proposition}\label{prop:BG} \cite{BryanGyenge22} Let \(Y\) be a symplectic resolution of \(S/G\). Then \(S^{[n,\tilde{\mathbf{m}}]}\), \(\tilde{\mathbf{m}}=i(\mathbf{m})\in M^G\),
  \(\mathbf{m}\in M_G\) is birational to
  \(Y^{[O(n,\mathbf{m})+1/2(\mathbf{m},\mathbf{m})]}\) where:
  \[(\mathbf{m},\mathbf{m})=\sum_{i=1}^{\ell}(\vec{m}_i,\vec{m}_i)_{\Delta_i},\,\,O(n,\mathbf{m})=(n-\sum_{i=1}^k\sum_{j=1}\tilde{m}_i^jd^j(\tilde{\Delta}_i))/|G|\]
\end{proposition}

For example, let us consider  \(S=\CC^2\) and \(G=\ZZ_2\). Then
\( (S^{[3]})^G=S^{[3;1]}\cup S^{[3;-1]}\), where \(\dim S^{[3;1]}=0\), \(\dim S^{[3;-1]}=2\).

The key lemma to the proposition is the local statement that relies on results by Nakajima \cite[Equation (2.6)]{Nak94}:

  \begin{lemma} Let \(S=\CC^2\). Then
    the scheme \(S^{[n,\vec{m}]}\subset (S^{[n]})^{G_\Delta}\) is a point if and only if  \(n=\vec{m}\cdot \vec{d}-|G|(\vec{m},\vec{m})_{\Delta}/2\).
  \end{lemma}
  
  The lemma implies that a generic point of a connected component of \((S^{[n]})^G\) corresponds to a union of \(k\) generic  \(G\)-orbits and
  the rigid clusters, as in the lemma, at the orbifold points. Hence the component is birational to \(\tilde{S}^{[k]}\)  where \(\tilde{S}\) is a resolution of
  \(S/G\).
  
  Thus, the number of connected components of \((S^{[n]})^G\) that are diffeomorphic to \(S^{[k]}\) is equal to the number of vectors \(\mathbf{m}\in M_G\)
  such that:
  \[k=O(n,\mathbf{m})+1/2(\mathbf{m},\mathbf{m}),\quad n-\sum_{i=1}^\ell\sum_{j=1}^{r_i}m_i^jd^j_i(\Delta_i)|G|/|G_i|= |G|\cdot O(n,\mathbf{m}).\]

Let us set notations for the standard theta functions:
\[\vartheta_{\Delta}(\mathbf{z};q)=\sum_{\vec{m}\in M_{\Delta}} \mathbf{z}^{\vec{m}}q^{-(\vec{m},\vec{m})_{\Delta}/2},\]
where \(M_{\Delta}\) is the root lattice of \(\Delta\) and \((\vec{m},\vec{m})_{\Delta}\) is the corresponding negatively defined quadratic form.

  It is convenient to assemble the corresponding numbers in a generating function:
  \begin{equation}\label{eq:theta}
    \Theta_G(q)=\prod_{i\in Sing(S/G)}\vartheta_{\Delta_i}(q^{d^1(\Delta_i)|G|/|G_i|},\dots, q^{d^{r_i}(\Delta_i)|G|/|G_i|} ;q^{|G|}).  
  \end{equation}
    Thus the number of \(k\)-dimensional components in \((S^{[n]})^G\) is a coefficient \(\Theta(n-k|G|)\) in front of \(q^{n-k|G|}\) in \(\Theta_G\):
  \[\Theta_G(q)=\sum_i \Theta_G[i]q^i.\]

\begin{theorem}\label{thm:mainK3}
Let $S$ be a K3 surface and let $G$ be a finite symplectic group of 
automorphisms of $S$. Let $S^{[n]}$ be the Hilbert scheme of $n$ points 
on $S$ and let us consider the induced action of $G$ on $S^{[n]}$. 
Then all the irreducible components of the locus stabilized by $G$ 
are deformation equivalent to Hilbert schemes of points on a K3 surface or isolated points, 
and their number for each dimension $2k$ is 
$$ N_k= \Theta_G[n-|G|k],$$
where $\Theta_G(k)$ is the $k$-th coefficients of the Theta series 
\eqref{eq:theta} of the lattice \(M^G\).
\end{theorem}

\begin{proof}
  By proposition~\ref{prop:BG} the coefficient \(N_k\) computes the number 
of components of the fixed locus that are birational to the Hilbert scheme
  of \(k\) points on \(S\). Each of these components are therefore also deformation equivalent to Hilbert schemes of \(k\) points on a \(K3\) surface, hence the claim follows. 
\end{proof}



\begin{exam} If \(G=C_2 = \ZZ_2\), the corresponding function factors 
in the following way:
  \[\Theta_{C_2}(q)=\vartheta(q;q^2)^8=1+8q+28q^2+64q^3+126q^4+\dots,\quad \vartheta(\eta;q)=\sum_{n\in \mathbb{Z}} q^{n^2}\eta^n.\]

  For the other groups the expression is more involved. For instance, 
if \(G=C_3 = \ZZ_3\) we have

\[\Theta_{C_3}(q)=\vartheta_{A_2}(q,q;q^3)^6=1+6q+27q^2+80q^3+\dots, \quad \vartheta_{A_2}(\eta;q)=\sum_{n_1,n_2\in \ZZ} q^{n_1^2+n_2^2-n_1n_2}\eta_1^{n_1}\eta_2^{n_2}.\]

For \(G=C_4=\ZZ_4\) we have:
\[\Theta_{C_4}(q)=\vartheta_{A_3}(q;q^4)^4\vartheta_{A_1}(q^2;q^4)=1+4q+16q^2+48q^3+118q^5+272q^6+\dots.\]

For \(G=C_2\times C_2=\ZZ_2\times\ZZ_2\) we have
  \[\Theta_{C_2\times C_2}(q)=\vartheta(q^2;q^4)^{12}=1+12q^2+66q^4+232q^6+627q^8+\dots.\]

  For \(G=C_5=\ZZ_5\) we have:
  \[\Theta_{C_5}(q)=\vartheta_{A_4}(q;q^5)^4=1+4q+14q^2+40q^3+105q^4+232q^5+494q^6+\dots.\]

  For \(G=C_6=\ZZ_6\) we have:
  \begin{multline*}
    \Theta_{C_6}(q)=\vartheta_{A_5}^2(q,q^6)^2\vartheta_{A_2}(q^2,q^6)^2\vartheta_{A_1}(q^3,q^6)^2=1+2q+7q^2+16q^3+39q^4+80q^5\\
    +151q^6+288q^7+\dots.\end{multline*}

  For \(G=D_8\) we have 
  \[\Theta_{D_8}(q)=\vartheta_{A_3}(q^2;q^{8})^2\vartheta_{A_1}(q^4;q^{8})^9=1+2q^2+14q^4+28q^{6}+93q^{8}+182q^{10}+406q^{12}+\dots\]
  
  For \(G=D_{10}\) we have:
  \begin{multline*}\Theta_{D_{10}}(q)=\vartheta_{A_5}(q^2,q^{10})\vartheta_{A_2}(q^5,q^{10})^8=1+2q^2+5q^4+8q^5+10q^6+16q^7+20q^8+40q^9+\\
      54q^{10}+80q^{11}+101q^{12}+160q^{13}+200q^{14}+\dots.
  \end{multline*}

\end{exam}

Let us denote by \(p=p(G)\) the least common multiple of the orders of stabilizer
subgroups of \(G\) acting on the corresponding K3 surface.
From examples above we see that the coefficients of the \(q\) expansion of
\(\Theta_G(q)\) are powers of \(q^p\). It turns out that the observation is true for all
groups \(G\) acting on a K3.

Next let us observe that the \(q\)-expansion of  \(\Theta_{D_{10}}(q)\) has vanishing coefficients in front of \(q\) and \(q^3\).
That is \(H^{4m}( (K3^{[10m+1]})^{D_{10}},\QQ)=H^{4m}((K3^{[10m+3]})^{D_{10}},\QQ)=0\) and
the dimension of \( (K3^{[10m+1]})^{D_{10}}\) and \((K3^{[10m+3]})^{D_{10}}\) are at most
\(2(m-1)\). On the other hand, by looking at the coefficients in front of \(q^{11}\)
and \(q^{13}\) we discover that the top dimensional pieces of \( (K3^{[10m+1]})^{D_{10}}\) and \((K3^{[10m+3]})^{D_{10}}\) consist of \(80\) and respectively \(160\) copies of
\(K3^{[m-1]}\)-type varieties.

 Using the theta function \(q\)-series we can compute the number of top dimensional 
 components in \((K3^{[n]})^G\) for all groups \(G\). We present the results of the computer program 
 in the table below.
Let us comment on the format of the table.

The number \(p=p(G)\) in the third column is exactly the characteristic of the group \(G\) from the above discussion.
In particular, the locus with maximal stabilizer \((K3^{[n]})^G\) is empty if \(p\) does not divide \(n\).


Now suppose \(n=m|G|+kp\), \(0\le k<|G|/p\), then it is natural to expect that 
the dimension of
\((K3^{[n]})^G\) is \(2m\). Indeed, it is the case for all groups in the table, 
except for the groups \(D_{10},A_4,S_4,A_5\). For the last groups, there are 
values of \(k\) for which \(\dim (K3^{[n]})^G=2m-2\).

The last column of the table lists the number of top-dimensional components.
For example, if we look at the line \(16\) that corresponds to the group \(D_{10}\),
then we see that the first line in the last column tells us that the number of
\(2m\)-dimesional components in \((K3^{[10m+k]})^{D_{10}}\) as \(k\) runs from \(0\) to \(9\). On the other hand if want to know how many \(2m-2\)-dimensional components
\((K3^{[10m+1]})^{D_{10}}\) and 
\((K3^{[10m+3]})^{D_{10}}\) have then we need to look at the second line of the entry.

\begin{table}[h!]\label{tab:comps}
\begin{tabular}{|c|c|c|c|c|}
\hline
 &  &  \(p\)&\(\epsilon\) &  \(\dim(H^{4(m-\epsilon)}((K3^{[m|G|+kp]})^G,\mathbb{Q}))) \), \quad \(0\le k<|G|/p\) \\\hline
1 & $C_2$ &  1 &0& 1,8\\ \hline
2 & $C_3$ &  1 &0 &1,6,27\\ \hline
3 & $C_2^2$ &  2 &0& 1,12\\ \hline
4 & $C_4$ &  1 & 0&1,4,16,48\\ \hline
5 & $C_{5}$ &  1 & 0& 1,4,14,40,105 \\ \hline
7 & $C_{6}$ &  1 & 0&1,2,7,16,39 \\ \hline

10 & $D_8$ &   2 &0& 1,2,14,28 \\ \hline
  16 & $D_{10}$ & 1 &0& 1,0,2,0,5,8,10,16,20,40 \\
     &         &   & 1& \(\star\),80,\(\star\),160,\dots        \\\hline
  17 & $A_4$ &  2 &0& 1,0,6,4,27,24\\ 
    &        &    &1& \(\star\),108,\dots\\ \hline
18 & $D_{12}$ &  2&0&  1,1,3,13,18,39 \\ \hline
  34 & $S_4$ &  2 &0& 1,0,0,2,3,0,7,6,9,14,21,18 \\
     &       &    &1               & \(\star\),42,63,\(\star\),\(\star\),126,\dots \\\hline
  55 & $A_5$ &  1 & 0& 1,0,0,0,0,0,2,0,0,0,3,0,5,0,0,4,6,0,10,0,9,8,15,0,20,12,18,20,30,0\\
  & & & 1& \(\star\),24,45,40,60,36,\(\star\),60,90,80,\(\star\),72,\(\star\),120,180,\(\star\),\(\star\),180,\(\star\),240,\(\star\),\(\star\),\(\star\),360,\dots,720\\ \hline

\end{tabular}
\vspace{10pt}

\caption{Number of connected components of the top dimension.}
\end{table}

\newpage

\subsubsection{Kummer case}
\label{sec:kummer-case}

Let \(G\subset \Aut_s(A)\) be a finite subgroup of symplectic automorphisms of  \(A\) that commutes 
with the Albanese map \(\Sigma:A^{[n]}\to A\) and preserves its zeroth fiber. In particular, the group 
\(G\) is a  subgroup of the automorphism group of \(A\).
Since the elements of \(G\) are group automorphisms of \(A\), then the image of \(\Sigma\) applied to \((A^{[n]})^G\) is one of the fixed points 
\(A^G\). 

It was pointed to us by an anonymous referee that the statement holds in bigger generality. We are thankful for the suggestion and provide the
argument and a statement as it was proposed to us.

\begin{lemma}
  Let \(G\) be a subgroup of symplectic automorphisms of \(A\). Then the Albanese map
  \(\Sigma: A^{[n]}\to A\) is constant on the connected components of \((A^{[n]})^G\).
\end{lemma}
\begin{proof}
  Let \(g\) be a
non-trivial symplectic automorphism of finite order inducing a symplectic
automorphism on \(K_{n-1}(A)\). Suppose that \(A = \mathbb{C}^2/\Gamma\) and
\(g(x) = Mx + b_1\),
where \(M \in \mathrm{SL}(\Gamma)\) and \(b_1 \in \mathbb{C}^2\) which induces a translation of order dividing
\(n + 1\) in \(A\). Write \(g^k(x) = M^kx + b_k\) with \(b_k :=
\sum^{k-1}_{j=0} M^j(b_1)\). Now, let’s
represent a point in the symmetric product \(A^{(n)}\) by a \(n\)-uple \((x_1,\dots , x_n)\).
A general point of a connected component  \(Z\) of the \(g\)-fixed locus \((A(n))^g\) is
a collection of \(r\) orbits (possibly of different length)
\begin{gather*}p = (x_1,Mx_1 + b_1, \dots ,M^{k_1-1}x_1 + b_{k_1-1}, x_2,Mx_2 + b_1, \dots ,M^{k_2-1}x_2 + b_{k_2-1}, \dots,
\\x_r,Mx_r + b_1, \dots ,M^{k_r-1}x_r + b_{k_r-1}),
\end{gather*}
where \(k_l\) is the length of the orbit of \(x_l\). Now, by construction, \([M^{k_l}x_l+b_{k_l} ]\)
and \([x_l]\) define the same point in \(A\), so
\[M^{k_l}x_l- x_l = (M -1)\sum^{k_l-1}_{t=0}
M^t x_l\]
belongs to a discrete set of \(\mathbb{C}^2\) (if \(b_1\) has order \(m\) in \(A\), then this discrete
set could be the preimage in \(\mathbb{C}^2\) of the \(m\)-torsion of \(A\) ), so it is constant
along \(Z\). Since \(g\) is a non-trivial symplectic automorphism of finite order, \(
M\) does not have eigenvalue \(1\), so \(M - 1\) is invertible, so
\[\sum^{k_l-1}_{t=0} M^tx_l = \alpha_l\]
is constant along  \(Z\). We conclude that \(\Sigma(p) =
\sum^r_{l=1}(\alpha_l+
\sum^{k_l-1}_{t=1} b_t)\) for any
point \(p \in Z\), so \(Z\) is entirely contained in a fiber of \(\Sigma\).
\end{proof}

Thus, we have a natural relation between  connected components of
\((A^{[n]})^G\) and connected components of \((K_{n-1}(A))^G\). Therefore, we can use the methods from the previous subsection to describe the connected components of \((K_{n-1}(A))^G\). Indeed, let \(Y=X/G\) be the corresponding singular \(K3\)-surface. The surface \(Y\) has orbifold singularities. That is, for \(z\in Sing(Y)\) there is \(G_z\subset G\) such that \(z\) is a singularity of type \(\CC^2/G_z\). Let \(\Delta_i\) be the root system of MacKay type \(G_{z_i}\), and let \(M_{\Delta_i}\) be the corresponding root lattice. Then, we define \(M\) and \(M^G\) as in \eqref{eq:MM}. 
The same argument as before implies:

\begin{proposition} Let \(Y\) be a symplectic resolution of \(A/G\). Then \(A^{[n,\mathbf{m}]}\), \(\mathbf{m}\in M^G\), is birational to
  \(Y^{[O(n,\mathbf{m})+1/2(\mathbf{m},\mathbf{m})]}\), where:
  \[(\mathbf{m},\mathbf{m})=\sum_{i=1}^k(\vec{m}_i,\vec{m}_i)_{\Delta_i},\,\,O(n,\mathbf{m})=(n-\sum_{i=1}\sum_{j=1}m_i^jd_j(\Delta_i))/|G|.\]
\end{proposition}

For enumerating connected components of \(K_{n-1}(A)^{G}\) we need to select the connected components of \((A^{[n]})^G\) that are inside \(K_{n-1}(A)^G\). To do that we introduce a
refinement of the theta function by the elements of \(A^G\).

In more details, an element \(\vec{m}\in M^G\) is a \(G\)-invariant sequence
of vectors  \(\vec{m}_i\), \(i=1,\dots,k\). Here \(k\) is the number points \(p_i\in A\) that have non-trivial \(G\)-stabilizer.
A point \(q_i\in Sing(A/G)\) corresponds to a orbit \(G(p_i)\) of some point \(p_i\in A\)  with non-trivial \(G\)-stabilizer.
We set \(\gamma_i=\prod_{\gamma\in G(p_i)}\gamma\in A\).

Thus, the  number of connected components of \((A^{[n]})^G\) is enumerated by the generating function \(\Theta_G(q)\) constructed from \(M^G\) by a modification of  formula \eqref{eq:theta}. Indeed, we define an element \(\Theta_G(q)\) of \(\mathbb{Z}[A^G][[q]] \) as
product
\eqref{eq:theta}:
\begin{equation}\label{eq:theta2}
    \Theta_G(q)=\prod_{i\in Sing(A/G)}\vartheta_{\Delta_i}(  \gamma_i \cdot q^{d^1(\Delta_i)|G|/|G_i|},\dots, \gamma_i\cdot q^{d^{r_i}(\Delta_i)|G|/|G_i|};(\gamma_i)^{|G_i|}\cdot (q)^{|G|}) 
  \end{equation}
  Then the number of components of dimension \(2k\) inside \(K_{n-1}(A)^G\) is equal to \(\Theta_G(n-|G|k;1)\):
  \begin{equation}\label{eq:theta2-coeffs}
  \Theta_G(q)=\sum_{i,\gamma\in A^G}\Theta_G[i;\gamma]q^i\cdot \gamma.\end{equation}

\begin{theorem} \label{thm:kum}
Let $X=K_{n-1}(A)$ be a  \(n\)-Kummer \hk manifold  and let $G\subset Aut(A)$ 
be a finite group of symplectic automorphisms of the abelian surface $A$.
Assume $G$ preserves the Albanese map of the generalized Kummer $X$, 
and is a regular group. 
Then all irreducible components of $X^G$ are of $K3^{[k]}$ type, and their 
number is 
$$ N_k= \Theta_G[n-|G|k;1].$$
\end{theorem}

\begin{proof}
  The coefficient \(\Theta_G[n-k|G|;\star]\) of \(\Theta_G\) in front of \(q^{n-k|G|}\) is an element of the group algebra of \(A^G\).
  Let \(|\Theta_G[n-k|G|;\star]|\) be a sum of  coefficients of this element. 
  By the above argument \(|\Theta_G[n-k|G|,\star]|\)  computes the number of \(2k\)-dimensional connected components of
  \((A^{[n]})^G\). Each of these components is birational, and hence deformation-equivalent, to a Hilbert scheme of points on a \(K3\) surface.
  Finally, let us observe that the coefficient in front of \(\gamma\in A^G\) in \(\Theta_G[n-k|G|;\star]\) computes the number of connected components of
  \((\Sigma^{-1}(\gamma))^G\).
\end{proof}

In the examples below \(A=E\times E\) and \(\delta: E\to A\) is the diagonal embedding.
We enumerate the connected components of \(K_{n-1}(A)\) in the next series of examples.
For an element \(z\)  of the group ring of \(A\)  we use notation \([1]z\) for the coefficient in
front of \(1\in A\).

\begin{exam} If \(G=\ZZ_2\), then
  \[\Theta_G(q)=\prod_{\gamma\in A:\gamma^2=1}\vartheta_{A_1}(\gamma\cdot q;q^2),\]
  \[[1]\Theta_G(q)=1+q+36q^3+140q^4+378q^5+1024q^6+\dots\]

 If \(G=\ZZ_3\), then  
 \[\Theta_G(q)=\prod_{\gamma\in E:\gamma^3=1}\vartheta_{A_2}(\delta(\gamma)\cdot q;q^3),\]
 \[[1]\Theta_G(q)=1+q+6q^2+12q^3+88q^4+255q^5+738q^6+\dots.\]

 If \(G=\ZZ_4\), then
  \[\Theta_G(q)=\prod_{i=1}^3\vartheta_{A_1}(\delta(\gamma_i)\cdot q^2;q^4)^2
    \prod_{i=0}^4\vartheta_{A_3}(\delta(\gamma_i)\cdot q;q^4),\]
  where \(\gamma_0=1\) and \(\gamma_i\), \(i>0\) are non-trivial second order points of \(E\). In particular, we have
  \[[1]\Theta_G(q)=1+q+8q^2+13q^3+35q^4+80q^5+147q^6+\dots\]

  If \(G=\ZZ_6\), then
  \[\Theta_G(q)=\vartheta_{A_5}(q;q^6)\vartheta_{A_2}(q^2;q^6)^4\vartheta_{A_1}(q^3;q^6)^5,\]
  \[[1]\Theta_G(q)=1+q+6q^2+12q^3+32q^4+63q^5+126q^6+\dots\]
\end{exam}

\begin{remark}
  In the examples above \(G=\ZZ_\ell\), \(\ell=2,3,4,6\). These are exactly the groups
  that are admissible.
\end{remark}

\begin{remark}
  The argument in the enumerative part of our previous paper
  \cite{KamenovaMongardiOblomkov19} contains a gap.
  The gap does not affects the enumeration of the top dimensional components of the fixed locus but the statements from \cite{KamenovaMongardiOblomkov19} about the lower-dimensional components are not correct.
  The results of the current paper provide a correction for these statements.
\end{remark}

\section{Main results}
Using the computations of section three, we are able to determine the fixed 
locus for all standard automorphisms, and we obtain the following:

\begin{theorem}
Let $X$ be a \hk manifold of \kntiposp and let $G\subset Aut(X)$ be a 
finite group of symplectic automorphisms among those listed in Table 1 
such that it is numerically standard in cohomology.
Then all irreducible components of $X^G$ are of $K3^{[k]}$ type, and their 
number is
$$ N_k= \Theta_G[n-|G|k].$$
\end{theorem}

\begin{proof}
By Proposition \ref{prop:goodgroups}, all numerically standard groups 
listed in Table 1 are standard, therefore we can deform 
the pair $(X,G)$ to a natural pair $(S^{[n]},G)$ and the result follows 
immediately from Theorem \ref{thm:mainK3}.
\end{proof}

\begin{theorem}\label{thm:main_special_dimension}
Let $X$ be a \hk manifold of \kntiposp such that $n-1$ is a prime power. 
Let $G\subset Aut(X)$ be a finite group of symplectic automorphisms which 
is numerically standard in cohomology.  
Then all irreducible components of $X^G$ are of $K3^{[k]}$ type, and their 
number is
$$ N_k= \Theta_G[n-|G|k],$$
where $\Theta_G(k)$ is the $k$-th coefficients of the Theta series 
\eqref{eq:theta} of the lattice \(M^G\).
\end{theorem}

\begin{proof}
As $n-1$ is a prime power, any numerically standard group in cohomology is  
standard by Theorem \ref{thm:main_old}, therefore we can deform the pair 
$(X,G)$ to a natural pair $(S^{[n]},G)$ and the result follows immediately 
from Theorem \ref{thm:mainK3}.
\end{proof}

\begin{theorem} \label{main-Kummer}
  Let $X$ be a \hk manifold of \(n\)-Kummer type and let $C_\ell=G\subset Aut(A)$,
  \(\ell=2,3,4,6\)
be a finite group of symplectic automorphisms of the abelian surface $A$.
Assume $G$ preserves the Albanese map of the generalized Kummer $X$, 
and acts non-trivially on $H^3(X)$. 
Then all irreducible components of $X^G$ are of $K3^{[k]}$ type, and their 
number is defined by \eqref{eq:theta2} and \eqref{eq:theta2-coeffs}:
$$ N_k= \Theta_G[n-|G|k;1].$$
\end{theorem}

\begin{proof}
  The groups \(C_\ell\), \(\ell=2,3,4,6\) are admissible, see the table in 
the appendix.  Hence the combination of the results of section~\ref{sec:adm} 
and theorem~\ref{thm:kum} implies this theorem.
\end{proof}


\section*{Appendix 1: Xiao's classification of groups acting on K3s.}
For ease of reference, we collect here the classification results obtained by Xiao \cite{xiao} on automorphism groups acting on $K3$ surfaces, together with the computations of the invariant lattice obtained by Hashimoto \cite{hashi}.

\newpage

\begin{table}[h!]
\begin{tabular}{|c|c|c|c|}
\hline
Label & Group & $T_G(K3)$& singularities in the quotient\\ \hline
1 & $C_2$ & $U^3\oplus E_8(-2)$& \(8\, A_1\)\\ \hline
2 & $C_3$ & $U\oplus U(3)\oplus A_2(-1)^2$& \(6\,A_2\)\\ \hline
3 & $C_2^2$ & $U\oplus U(2)^2\oplus D_4(-2)$& \(12 A_1\)\\ \hline
4 & $C_4$ & $U\oplus U(4)^2\oplus -2^2$& \(4A_3 + 2A_1\) \\ \hline
5 & $C_{5}$ & $U\oplus U(5)^2$& \(4A_4\)\\ \hline
6 & $D_6$ & $U(3)\oplus A_2(2)\oplus A_2(-1)^2$ & $8A_1+3A_2$\\ \hline
7 & $C_{6}$ & $U\oplus U(6)^2$& \(2A_5+2A_2+2A_1\)\\ \hline
8 & $C_7$ & $U(7)\oplus \left(\begin{array}{cc} 2 & 1 \\ 1 & 4 \end{array}\right)$ & $3A_6$\\\hline
9 & $C_2^3$ & $U(2)^3\oplus (-4)^2$ & $14A_1$   \\ \hline
10 & $D_8$ & $U\oplus (4)^2 \oplus (-4)^3$& \(2A_3+9A_1\) \\ \hline
11 & $C_2\times C_4$ & $U(2)\oplus (4)^2\oplus (-4)^2$ & $4A_1+4A_3$   \\ \hline
12 & $Q_8$ & $\left(\begin{array}{ccc} 6 & 2 & 2\\ 2& 6 & -2\\ 2& -2 &6 \end{array}\right)\oplus (-2)^2$ & $3A_3 + 2 D_4$ \\ \hline
13 & $Q_8$ & $(4)^3\oplus (-4)^2$ & $A_1 + 4 D_4$ \\ \hline
14 & $C_8$ & $U(8)\oplus (2)\oplus (4)$ & $A_1 + A_3 + 2 A_7$ \\ \hline
15 & $C_3^2$  & $U(3)^2\oplus \left(\begin{array}{cc} 2 & 3 \\ 3 & 0 \end{array}\right)$ & $8A_2$ \\ \hline
16 & $D_{10}$ & $U\oplus U(5)^2$& \(2A_4+8A_1\)\\ \hline
17 & $A_4$ & $U\oplus A_2(2)\oplus A_2(-4)$& \(6A_2+4A_1\) \\ \hline
18 & $D_{12}$ & $U\oplus U(6)^2$& \(A_5+A_2+9A_1\)\\ \hline
19 & $C_2 \times C_6$ & $U(3)\oplus A_2(4)$ & $3A_1+3A_5$ \\ \hline
20 & $Q_{12}$ & $U(3)\oplus A_2(4)$ & $A_2+2A_3+2D_5$ \\ \hline
21 & $C_2^4$ & $U(2)^3\oplus (-8)$ & $15A_1$ \\ \hline
22 & $C_2 \times D_8$ & $U(2)\oplus (4)^2\oplus (-4)^2$ & $10A_1 + 2A_3$ \\ \hline
23 & $\Gamma_2 c_1$ & $U(2)\oplus (4)\oplus (-4)\oplus (8)$ & $5A_1+4A_3$ \\ \hline
24 & $Q_8 * C_4$ & $(4)^3\oplus (-4)^2$ & $6A_1 + A_3 +2D_4$ \\ \hline
25 & $C_4^2$ & $ \left(\begin{array}{cccc} 4& 0 & 2 &0\\ 0&4& 2&0\\ 2& 2&4&4\\ 0&0&4&0 \end{array}\right)$ & $6A_3$ \\ \hline
26 & $SD_{16}$ & $U(8)\oplus (2)\oplus (4)$ & $4A_1 + A_3 + A_7 + D_4$ \\ \hline
27 & $C_2 \times Q_8$ & $ \left(\begin{array}{cccc} 4& 0 & 2 &0\\ 0&4& 2&0\\ 2& 2&4&4\\ 0&0&4&0 \end{array}\right)$ & $2A_1 + 4D_4$ \\ \hline

\end{tabular}
\end{table}

\newpage

\begin{table}[h!]
\begin{tabular}{|c|c|c|c|}
\hline
Label & Group & $T_G(K3)$& singularities in the quotient\\ \hline
28 & $\Gamma_2 d$ & $(4) \oplus (8)^2$ & $2A_1+A_3+2A_7$ \\ \hline
29 & $Q_{16}$ & $(4) \oplus (8)^2$  & $A_3 + D_4 + 2D_6$  \\ \hline
30 & $A_{3,3}$ & $U(3)^2\oplus \left(\begin{array}{cc} 2 & 3 \\ 3 & 0 \end{array}\right)$ & $8A_1 + 4A_2$ \\ \hline
31 & $C_3 \times D_6$ & $U(3)\oplus A_2(6)$ &  $2A_1 + 3A_2 + 2A_5$ \\ \hline
32 & $Hol(C_5)$ & $U(5)\oplus \left(\begin{array}{cc} 4 & 2 \\ 2 & 6 \end{array}\right)$ & $2A_1 + 4A_3 + A_4$ \\ \hline
33 & $C_7 \rtimes C_3$ & $U(7)\oplus \left(\begin{array}{cc} 2 & 1 \\ 1 & 4 \end{array}\right)$ & $6A_2 + A_6$ \\ \hline
34 & $S_4$ & $U\oplus A_2(4)\oplus -12$& \(2A_3+3A_2+5A_1\) \\ \hline
35 & $C_2 \rtimes A_4$ & $U(2)\oplus (12)^2$ & $4A_1 + 2A_2 + 2A_5$ \\ \hline
36 & $C_3 \rtimes D_8$ & $U(3)\oplus A_2(4)$ & $5A_1 + A_3 + A_5 + D_5$ \\ \hline
37 & $T_{24}$ & $ \left(\begin{array}{ccc} 4 & 0 &0\\ 0&8&4\\ 0&4&8 \end{array}\right)$ & $2A_2 + A_5 + D_4 + E_6$ \\ \hline
38 & $T_{24}$ & $ \left(\begin{array}{ccc} 2 & 0 &0\\ 0&16&8\\ 0&8&16 \end{array}\right)$ & $2A_2 + A_3 + 2E_6$ \\ \hline
39 & $2^4 C_2$  & $U(2)\oplus (4)\oplus (-4)\oplus (8)$ & $8A_1 + 3 A_3$ \\ \hline
40 & $Q_8 * Q_8$ & $(4)^3\oplus (-4)^2$ & $9A_1 + 2D_4$ \\ \hline
41 & $\Gamma_7 a_1$ & $(4)^3\oplus (-8)$ & $3A_1 + 5A_3$ \\ \hline

42 & $\Gamma_4 c_2$ & $ \left(\begin{array}{cccc} 4& 0 & 2 &0\\ 0&4& 2&0\\ 2& 2&4&4\\ 0&0&4&0 \end{array}\right)$ & $2D_4+2A_3+4A_1$  \\ \hline
43 & $\Gamma_7 a_2$ & $(4) \oplus (8)^2$ & $2A_7+5A_1$  \\ \hline
44 & $\Gamma_3 e$ & $(4) \oplus (8)^2$ & $D_4+A_7+2A_3+2A_1$  \\ \hline
45 & $\Gamma_6 a_2$ & $(4) \oplus (8)^2$ & $ 2D_6+D_4+3A_1$  \\ \hline
46 & $3^2C_4$ & $A_2\oplus (6)\oplus (-18)$ & $4A_3+2A_2+2A_1$  \\ \hline
47 & $C_3\times A_4$ & $U(3)\oplus A_2(4)$ & $A_5+6A_2+A_1$  \\ \hline
48 & $S_{3,3}$ & $U(3)\oplus A_2(6)$ & $ 2A_5+A_2+6A_1+$  \\ \hline
49 & $2^4C_3$ & $U(2)\oplus A_2(2)\oplus (-8)$ & $ 6A_2+5A_1$  \\ \hline
50 & $4^2C_3$ & $ \left(\begin{array}{cccc} 4& 0 & 2 &0\\ 0&4& 2&0\\ 2& 2&4&4\\ 0&0&4&0 \end{array}\right)$ & $ 2A_3+6A_2$  \\ \hline
51 & $C_2\times S_4$ & $U(2)\oplus (12)^2$ & $A_5+2A_3+A_2+5A_1$  \\ \hline

\end{tabular}
\end{table}

\newpage

\begin{table}[h!] 
\begin{tabular}{|c|c|c|c|}
\hline
Label & Group & $T_G(K3)$& singularities in the quotient\\ \hline

52 & $ 2^2(C_2\times C_6)$ & $ \left(\begin{array}{ccc} 8 & 4 &4\\ 4&8&2\\ 4&2&8 \end{array}\right)$ & $3A_5+4A_1$  \\ \hline
53 & $2^2Q_{12}$ & $ \left(\begin{array}{ccc} 8 & 4 &4\\ 4&8&2\\ 4&2&8 \end{array}\right)$ & $2D_5+2A_3+A_2+A_1$  \\ \hline
54 & $T_{48}$ & $ \left(\begin{array}{ccc} 2 & 0 &0\\ 0&16&8\\ 0&8&16 \end{array}\right)$ & $E_6+A_7+A_2+4A_1$  \\ \hline
  55 & $A_5$ & $U\oplus A_2(10)$& \(2A_4+3A_2+4A_1\)\\ \hline
56 & $\Gamma_{25} a_1$ & $(4)^3\oplus (-8)$ & $D_4+3A_3+5A_1 $  \\ \hline
57 & $\Gamma_{13}a_1$ & $ \left(\begin{array}{cccc} 4& 0 & 2 &0\\ 0&4& 2&0\\ 2& 2&4&4\\ 0&0&4&0 \end{array}\right)$ & $3D_4+6A_1$  \\ \hline
58 & $\Gamma_{22}a_1 $ & $(4) \oplus (8)^2$ & $A_7+3A_3+3A_1 $  \\ \hline
59 & $\Gamma_{23}a_2$ & $(4) \oplus (8)^2$ & $D_4+5A_3$  \\ \hline
60 & $\Gamma_{26}a_2$ & $(4) \oplus (8)^2$ & $2D_6+A_3+4A_1$  \\ \hline
61 & $A_{4,3}$ & $U(3)\oplus A_2(4)$ & $D_5+A_3+3A_2+4A_1$  \\ \hline
62 & $N_{72}$ & $ \left(\begin{array}{ccc} 6 & 3 &3\\ 3&6&3\\ 3&3&12 \end{array}\right)$ & $2A_5+2A_3+3A_1$  \\ \hline
63 & $M_9$ & $ \left(\begin{array}{ccc} 2 & 0 &0\\ 0&12&6\\ 0&6&12 \end{array}\right)$ & $2D_4+3A_3+A_2$  \\ \hline
64 & $2^4C_5$ & $ \left(\begin{array}{ccc} 4 & 0 &2\\ 0&4&2\\ 2&2&12 \end{array}\right)$ & $4A_4+3A_1$  \\ \hline
65 & $2^4D_6$ & $A_2(2)\oplus (4)\oplus (-8)$ & $3A_3+3A_2+3A_1$  \\ \hline
66 & $2^4C_6$ & $(4)\oplus (8)\oplus (12)$ & $2A_5+A_3+2A_2+2A_1$  \\ \hline
67 & $4^2D_6$ & $(4) \oplus (8)^2$ & $D_4+A_7+3A_2+2A_1$  \\ \hline
68 & $2^3 D_{12}$  & $ \left(\begin{array}{ccc} 8 & 4 &4\\ 4&8&2\\ 4&2&8 \end{array}\right)$ & $D_5+A_5+2A_3+3A_1$ \\ \hline

\end{tabular}
\end{table}

\newpage

\begin{table}[h!]\label{tab:allgroups}
\begin{tabular}{|c|c|c|c|}
\hline
Label & Group & $T_G(K3)$& singularities in the quotient\\ \hline

69 & $(Q_8 * Q_8) \rtimes C_3$  & $ \left(\begin{array}{ccc} 4 & 0 &0\\ 0&8&4\\ 0&4&8 \end{array}\right)$ & $2E_6 + 2A_2 + 3A_1$ \\ \hline
70 & $S_5$ & $ \left(\begin{array}{ccc} 4 & 1 &0\\ 1&4&0\\ 0&0&20 \end{array}\right)$ or $ \left(\begin{array}{ccc} 4 & 2 &2\\ 2&6&1\\ 2&1&16 \end{array}\right)$ & $A_5+A_4+2A_3+A_2+2A_1$  \\ \hline
71 & $F_{128}$ & $(4) \oplus (8)^2$ & $D_6 + D_4 + 2A_3 + 3A_1$ \\ \hline
72 & $A_4^2$ & $ \left(\begin{array}{ccc} 8 & 4 &4\\ 4&8&2\\ 4&2&8 \end{array}\right)$ & $2A_5+4A_2+A_1$ \\ \hline
73 & $2^4D_{10}$ & $ \left(\begin{array}{ccc} 4 & 0 &2\\ 0&4&2\\ 2&2&12 \end{array}\right)$ & $2A_4+3A_3+2A_1$ \\ \hline
74 & $L_2(7)$ & $ \left(\begin{array}{ccc} 2 & 1 &0\\ 1&4&0\\ 0&0&28 \end{array}\right)$ or $ \left(\begin{array}{ccc} 4 & 2 &2\\ 2&8&1\\ 2&1&8 \end{array}\right)$ & $A_6+2A_3+3A_2+A_1$ \\ \hline
75 & $4^2A_4$ & $ \left(\begin{array}{cccc} 4& 0 & 2 &0\\ 0&4& 2&0\\ 2& 2&4&4\\ 0&0&4&0 \end{array}\right)$ & $D_4+6A_2+2A_1$ \\ \hline
76 & $H_{192}$ & $(4)\oplus (8)\oplus (12)$ & $D_4+A_5+2A_3+A_2+2A_1$ \\ \hline
77 & $T_{192}$ & $ \left(\begin{array}{ccc} 4 & 0 &0\\ 0&8&4\\ 0&4&8 \end{array}\right)$ & $E_6+3A_3+A_2+2A_1$  \\ \hline
78 & $A_{4,4}$ & $ \left(\begin{array}{ccc} 8 & 4 &4\\ 4&8&2\\ 4&2&8 \end{array}\right)$ & $2D_5+A_3+2A_2+2A_1$ \\ \hline
79 & $A_6$ & $ \left(\begin{array}{ccc} 2 & 1 &0\\ 1&8&0\\ 0&0&12 \end{array}\right)$ or $ \left(\begin{array}{ccc} 6 & 0 &3\\ 0&6&3\\ 3&3&8 \end{array}\right)$ & $2A_4+2A_3+2A_2+A_1$ \\ \hline
80 & $F_{384}$ & $(4) \oplus (8)^2$ & $D_6+2A_3+3A_2+A_1$  \\ \hline
81 & $M_{20}$ & $ \left(\begin{array}{ccc} 4 & 0 &2\\ 0&4&2\\ 2&2&12 \end{array}\right)$ & $D_4+2A_4+3A_2+A_1$ \\ \hline

\end{tabular}
\vspace{10pt}
\caption{Symplectic automorphisms of K3 surfaces}
\end{table}


\section*{Appendix 2: Fujiki's classification of automorphisms of abelian surfaces}

We briefly recall the classification of finite symplectic automorphisms of Abelian surfaces, initially worked out by Fujiki \cite{Fuj88}. From the classification, we exclude translations and look at the induced action on the second cohomology. As $-1$ always acts trivially on the second cohomology, we will always assume that $-1\in G$. We will consider the associated faithful action of $\widetilde{G}=G/\pm1$ on $H^2(A,\ZZ)$.  
As finite symplectic automorphisms preserve a K\"ahler class, the symplectic form and its conjugate, it follows that the invariant lattice $T_G(A):=H^2(A,\ZZ)^G$ has signature $(3,r)$ and its orthogonal complement, the coinvariant lattice $S_G(A)$, is negative definite and of rank at most three. From this we have the following straightforward proposition:
\begin{proposition}
Let $\widetilde{G}\subset O(H^2(A))$ be a finite group of symplectic isometries coming from symplectic automorphisms, then $\widetilde{G}\subset SO(E_8)$ and $S_{\widetilde{G}}(X)\cong S_{\widetilde{G}}(E_8)(-1)$.
\end{proposition}

The above proposition is analogous to \cite[Proposition 5.2]{MTW} in the generalized Kummer fourfolds case.
Moreover we have the following converse:
\begin{proposition}\label{prop:e8tosympl}
Let $\widetilde{G}\subset SO(E_8)$ be a group and suppose $S_{\widetilde{G}}(E_8(-1))$ has rank at most 3. Then $\widetilde{G}$ is induced by a group $G$ of symplectic automorphisms of some abelian surface $A$ such that $G/\pm 1=\widetilde{G}$. 
\end{proposition}

To conclude, we can use the following combinatorial criterion:

\begin{theorem}\cite[Thm. 3.6]{HM15}
Let $G$ be a subgroup of $O(E_8)$ which is the stabilizer of some sublattice of $E_8$. Then $G$ is the Coxeter group of a Dynkin sublattice of $E_8$ and $S_G(E_8)$ is the above said Dynkin lattice. 
\end{theorem}

In the following table (cf. Tables in section 5 of in \cite{MTW}), we list all these lattices up to rank 3, together with the group of automorphisms induced by their determinant one isometries. To denote the groups, we use the following notation: $n^m$ denotes the cartesian product of $m$ cyclic groups of order $n$, $Q_8$ denotes the quaterion subgroup, $T_{24}$ the binary tetrahedral group, and $D$ the binary dihedral group of order 12.

\vspace{5pt}
\begin{table}[h!]
\begin{tabular}{|c|c|c|c|c|}\hline
Rank $T_G(A)$ & $G$ & $T_G(A)$ & $S_G(A)$ & Singularities of the quotient\\\hline
2 & $C_4$ & $U\oplus A_1^2(-1)$ & $A_1^2$ & $6A_1+4A_3$\\\hline
2 & $C_6$ & $U\oplus A_2(-1)$ & $A_2$ & $5A_1+4A_2+A_5$ \\\hline
3 & $Q_8$ & $A_1(-1)^3$ & $A_1^3$ & $3A_1+4D_4$\\\hline
3 & $Q_8$ & $A_1(-1)^3$ & $A_1^3$ & $2A_1+3A_3+2D_4$\\\hline
3 & $D$ & $A_1(-1) \oplus A_2(-1)$ & $A_1 \oplus A_2$ & $A_1+2A_2+3A_4+D_5$ \\\hline
3 & $T_{24}$ & $A_3(-1)$ & $A_3$ & $A_1+4A_2+d_4+E_6$ \\\hline

\end{tabular}
\vspace{10pt}
\caption{Relevant finite symplectic automorphisms of Abelian surfaces}
\end{table}

An explicit geometric description of these case is also available, we recall it from \cite[Proposition 3.7]{Fuj88}:
For $G=C_4,C_6$, the surfaces can be deformed to a pair $E\times E$, where $E$ is an elliptic curve. The action of order four is given by $(e,f)\to (-f,e)$, while the action of order six is given by $(e,f)\to (e+f,-e)$.
For the quaternionic group, we can take the elliptic curve $E_i$ with an automorphism $\varphi$ of order four and consider the group of automorphisms of $E_i\times E_i$ generated by $(e,f)\to (-f,e)$; $(e,f)\to (\varphi(f),-e)$ and $(e,f)\to (-f,\varphi(e))$. The two cases for the group $Q_8$ are obtained from \cite[Theorem 3.11 and Remark 3.12]{Fuj88}, they correspond to two one-dimensional families of dual abelian surfaces.


\subsection*{Acknowledgements}
We would like to thank G. Oberdieck for comments on the paper.
The first named author is partially supported by a grant from the 
Simons Foundation/SFARI (522730, LK). 
The second named author was supported by PRIN2020 research grant
”2020KKWT53” and is a member of the INDAM-GNSAGA.
The third named author was supported by  NSF-FRG grant 1760373 and
NSF grant 2200798.

We thank the SCGP, Stony Brook University at which some of the research 
for this paper was performed during the hyperk\"ahler program in January/ 
February of 2023.


\end{document}